\documentclass{LMCS}

\def\dOi{10(1:3)2014}
\lmcsheading%
{\dOi}
{1--18}
{}
{}
{Feb.~11, 2013}
{Jan.~21, 2014}
{}

\ACMCCS{[{\bf Theory of computation}]: Models of computation - Computability - Lambda calculus}
\usepackage{enumitem,hyperref}

\theoremstyle{definition}
\newtheorem{coxa}[thm]{Comment and Example}
\newtheorem{noco}[thm]{Notations and Example}

\def\b{\beta}
\def\g{\gamma}

\def\r{\rho}

\def\l{\lambda}

\def\n{\nu}

\def\r{\rho}

\def\n{\nu}
\def\s{\sigma}

\def\O{\Omega}

\def\tr{\triangleright}
\def\trbo{\tr_{\b \Omega}}

\def\<{\langle}
\def\>{\rangle}

\def\N{\mathbb{N}}

\begin{document}

\title[About the range property for $\mathcal{H}$]{About the range property for $\mathcal{H}$}
%\titlecomment{{\lsuper*}OPTIONAL comment concerning the title, \eg, 
% if a variant or an extended abstract of the paper has appeared elsewhere.}

\author[R.~David]{Ren\'e David}	%required
\address{LAMA - Equipe LIMD - Universit\'e de Savoie - 73376 Le Bourget du Lac}	
\email{david@univ-savoie.fr -- nour@univ-savoie.fr}  %optional
%\thanks{thanks 1, optional.}	%optional

\author[K.~Nour]{Karim Nour}	%optional
%\address{LAMA - Equipe LIMD - Universit\'e de Savoie - 73376 Le Bourget du Lac}	
%\email{nour@univ-savoie.fr}  %optional
%\thanks{thanks 2, optional.}	%optional

%%%%%%%%%%%%%%%%%%%%%%%%%%%%%%%%%%%%%%%%%%%%%%%%%%%%%%%%%%%%%%%%%%%%%%%%%%%

%% the abstract has to PRECEDE the command \maketitle:
%% be sure not to issue the \maketitle command twice!

\begin{abstract}
Recently, A. Polonsky (see \cite{Pol12}) has shown that the range
property fails for  $\mathcal{H}$. We give here some conditions on
a closed term that imply that its range has an infinite
cardinality.
\end{abstract}

\maketitle

%% start the paper here:
\section*{Introduction}\label{S:one}

\subsection{Our motivations}\hspace*{\fill} \\

\noindent Let $\mathcal{T}$ be a $\l$-theory. The {\em range property} for
$\mathcal{T}$  states that if $\l x.F$ is a closed $\l$-term, then its
range (considering $\l x.F$ as a map from  $\mathcal{M}$ to
$\mathcal{M}$ where $\mathcal{M}$ is the algebra of closed
$\l$-terms modulo the equality defined by $\mathcal{T}$) has
cardinality either 1 or $Card(\mathcal{M})$.
%In \cite{Boh68}
%B\"ohm conjectured the range property for $\b$. The range property
%for the theory equating terms with the same B\"ohm tree is proved
%quite easily. B\"ohm's  original conjecture was proved
%independently by Myhill and Barendregt (see \cite{Bar85}).
It has been proved by Barendregt in \cite{Bar93} that  the range
property holds for all recursively enumerable theories. For the
 theory $\mathcal{H}$ equating all unsolvable terms, the
validity of the range property has been an open problem for a long
time. Very recently A. Polonsky has shown (see \cite{Pol12}) that
it fails.

In an old attempt to prove the  range property for $\mathcal{H}$,
Barendregt (see \cite{Bar08}) suggested a possible way to get the
result. The idea was, roughly, as follows. First observe that, if
the range of a term  is not a singleton, it will reduce to a term
of the form $\l x. F[x]$.  Assuming that, for some $A$, $F[x:=A]
\neq_{\mathcal{H}} F[x:=\O]$, he proposed another term $A'$ (the
term  $J_{\n}\circ A$ of Conjecture 3.2 in \cite{Bar08}) having a
free variable $\n$ that could never be erased or used  in a
reduction of $F[x:=A']$. He claimed that, by the properties of the
variable $\n$, the terms $F[x:=A_n]$ should be different where
$A_n=A'[\nu:= c_n]$ and $c_n$ is, for example, the Church numeral
for $n$.  It has been rather quickly understood by various
researchers that the term proposed in \cite{Bar08} actually {\em
had not} the desired property and, of course, Polonsky's result
shows  this method could not work. Moreover, even if the term $A'$
had the desired property it is, actually, not true that the terms
$F[x:=A_n]$ would be different: see section \ref{BP} below.

\subsection{Our results}\hspace*{\fill} \\

\noindent Even though the failure of the range property for $\mathcal{H}$ is
now  known, we believe that having conditions which imply that
the range of a term is infinite is, ``by itself'' interesting. We
also think   that the idea proposed by Barendregt remains
interesting and, in this paper,  we consider the following
problem.

Say that $F$ has the {\em Barendregt's persistence property} if for each $A$
such that $F[x:=A] \neq_{\mathcal{H}} F[x:=\O]$ we can find a term
 $A'$ that  has a free variable $\n$ that  could
never be erased or applied in a reduction of $F[x:=A']$ (see
definition \ref{applied}).

We give here some conditions on terms that
imply the Barendregt's persistence property. Our main result is Theorem
\ref{cases}. It can be stated as follows. Let $F$ be a term having
a unique free variable $x$ and $A$ be a closed term such that
$F[x:=A] \neq_{\mathcal{H}} F[x:=\O]$. We introduce a sequence
$(F_k)_{k \in \N}$ of reducts of $F$ that can simulate (see Lemma \ref{Fk})
all the reductions of $F$. By considering, for each $k$, the
different occurrences of $x$ in $F_k$ we introduce a tree ${\mathcal
T}$ and a special branch in it (see Theorem \ref{tree}). Denoting
by $x_{(k)}$ the corresponding occurrence of $x$ in $F_k$, Theorem
\ref{cases} states that:
\begin{enumerate}
\item  If, for $k \in \mathbb{N}$, the number of arguments of
$x_{(k)}$ in $F_k$ is bounded,  $F$ has the Barendregt's persistence property.

\item Otherwise and assuming the branch in ${\mathcal T}$ is recursive
there are two cases.

\begin{enumerate}[label=(2.\alph*)]
\item If some of the arguments of $x_{(k)}$ in $F_k$
come in head position infinitely often during the head reduction
of $F[x:=A]$, then $F$ has the Barendregt's persistence property.

\item Otherwise, it is possible that $F$ does not
have the Barendregt's persistence property.
\end{enumerate}
\end{enumerate}

\noindent In case (1) we give  two arguments.  The first one is quite easy
and uses this very particular situation. It gives a term $A'$
where $\n$ cannot be erased but we have not shown that it is never
applied (it is not applied only in the branch defined in section
2). The second one (which is more complicated) is a complete proof
that $F$ has the Barendregt's persistence property. It can also be seen as  an
introduction to the more elaborate case 2.(a). For case 2.(b) we
give examples of terms $F$ for which there is no term having this
property.

\subsection{A final remark}\hspace*{\fill} \\\label{BP}

\noindent Note that, actually, even if a term $A'$ {\em having}
the Barendregt's persistence property can be found this would not give an infinite
range for $\l x.F$. This is due to the fact that the property of
$A'$  does not imply that $F[x=A_n] \neq_{\mathcal{H}} F[x=A_m]$ for $n
\neq m$ where $A_k=A'[\n=c_k]$ and $c_k$ is
the Church integer for $k$. This is an old result of Plotkin (see \cite{Plo74}). Thus,
 having an infinite range will need another assumption.
We will also consider this other assumption and thus give simple
criteria that imply that the range of $\l x.F$ is infinite.

The paper is organized as follows. Section 1 gives the necessary
 definitions. Section 2 considers the possible situations
for $F$ and states our main result.   Section 3 and 4 give the
proof in the cases where we actually can find an $A'$.  Section 5
gives some complements.\newpage

\section{Preliminaries}

\begin{nota}\hfill
\begin{enumerate}
\item We denote by $\Lambda^{\circ}$ the set of closed $\l$-terms.
\item We denote by  $c_k$ the Church integer for $k$ and
by $Suc$ a closed term for the successor function.
 As usual we denote by $I$ (resp. $K$, $\Omega$) the term $\l
x. x$ (resp. $\l x\l y.  x$,  $(\delta \ \delta)$ where $\delta=
\l x. (x \ x)$).
\item $\#(t)$ is the code of $t$ i.e. an integer coding the way $t$ is built.

 \item We denote by $\tr$ the $\b$-reduction, by $\tr_{\Omega}$ the $\O$-reduction (i.e. $t \tr_{\Omega} \Omega$ if $t$ is
 unsolvable), by $\tr_h$ the head $\b$-reduction and by $\tr_{\b \Omega}$ the union of $\tr$ and $\tr_{\Omega}$.

%\item $t \tr t'$ means that $t$ reduces to $t'$ by one step of
%the $\tr_{\beta}$ or $\tr_{\Omega}$ reduction.
%\item If $R \in \{\tr_{\beta}, \tr_{\Omega}, \tr \}$, $t\ R^* \ t'$ means
%that $t$ reduces to $t'$ by some (possibly zero) of the $R$ reduction.
\item If $R$ is a notion of reduction, we denote by $R^*$ its
reflexive and transitive closure.
\item Let $x$ be a free variable of a term $t$. We denote by $x
\in \b\O(t)$ the fact that $x$ does occur in any $t'$ such that $t
\tr_{\b \Omega}^* t'$.
\item As usual  $(u \ t_1 \ t_2 \ ... \ t_n)$ denotes $( ...((u \ t_1) \ t_2)  \ ... \
t_n)$. $(u^n \ v)$ will denote $(u \ (u \ ... \ (u \ v)...))$ and
$(v \ u^{\sim n})$ will denote $(v \ u \ ... \ u)$ with $n$
 occurrences of $u$.
\end{enumerate}
\end{nota}

\begin{nota}\hfill
\begin{enumerate}
  \item Unknown sequences (possibly empty) of abstractions or
  terms will be denoted with an arrow. For example $\l
  \overrightarrow{z}$ or $\overrightarrow{w}$. However, to improve readability,
  capital letters will also be used  to denote sequences. The notable exceptions are $F,A,J$ taken
from Barendregt's paper or standard notations for terms as $I, K,
\O$. When the meaning is not clear from the context, we will
explicitly say something as ``the sequence $\l
  \overrightarrow{z}$ of abstractions".

  \item For example, to mean that a term $t$ can
be written as some abstractions followed by the application of the
variable $x$ to some terms, we will say that $t=\l
  \overrightarrow{z}. (x \ \overrightarrow{w})$ or $t=\l Z. (x \
  W)$.
  \item If $R,S$ are sequences of terms of the same length, $R
  \tr^* S$ means that each term of the sequence $R$ reduces to the
  corresponding term of the sequence $S$.

\end{enumerate}
\end{nota}

\begin{lem}\label{commutation}\hfill
\begin{enumerate}
  \item
If $u \trbo^* v$, then  $u \tr^* w \tr_{\Omega}^* v$ for some $w$.
  \item $\trbo^*$ satisfies the Church-Rosser property :
if $t \trbo^* t_1$ and $t \trbo^* t_2$, then  $t_1 \trbo^* t_3$ and $t_2
\trbo^* t_3$ for some $t_3$.
\end{enumerate}

\end{lem}

\begin{proof}
See, for example, \cite{Bar85}.
\end{proof}

%[Computable Church-Rosser property]

\begin{thm}\label{CR}
If $t \tr^* t_1$ and $t \tr^* t_2$, then  $t_1 \tr^* t_3$ and $t_2
\tr^* t_3$ for some $t_3$. Moreover $\#(t_3)$ can be computed from
$\#(t_1)$, $\#(t_2)$ and a code for the reductions $t \tr^* t_1$,
$t \tr^* t_2$.
\end{thm}

\begin{proof}
See \cite{Bar85}.
\end{proof}

\begin{defi}\hfill
\begin{enumerate}
\item We denote by $\simeq$ the equality modulo $\trbo^*$ i.e.
$u \simeq v$ iff there is $w$ such that $u \trbo^* w$ and $v \trbo^*
w$ and  we denote by $[t]_{\mathcal{H}}$ the class of $t$ modulo
$\simeq$.
\item For  $\l x.F \in \Lambda^{\circ}$, the range of $\l x.F$ in $\mathcal{H}$ is  the set
$\Im (\l x.F) = \{[F[x:= u]]_{\mathcal{H}}$ / $u \in \Lambda^{\circ}\}$.
\item A closed term $\l x.F$ has the range property for $\mathcal{H}$ if the set $\Im (\l x.F)$
is either infinite or has a unique element.
\end{enumerate}
\end{defi}

\begin{thm}\label{Andrew}
There is a term $\l x.F \in \Lambda^{\circ}$  that has not the range
property for $\mathcal{H}$.
\end{thm}

\begin{proof}
See \cite{Pol12}.
\end{proof}

\begin{defi}
Let $U, V$ be finite sequences of terms.
\begin{enumerate}
\item  We denote by $U::V$ the
list obtained by putting $U$  in front of $V$.
\item  $U \sqsubseteq V$ means that
some initial subsequence of $V $ is obtained from $ U$ by
substitutions and reductions.
\end{enumerate}
\end{defi}

\begin{lem}\label{transitive}\hfill
\begin{enumerate}
  \item $U \sqsubseteq
V$ iff there is a substitution $\s$ such that $\s(U)$ reduces to
an initial segment of $V$.
  \item The relation $\sqsubseteq $ is transitive.
\end{enumerate}
\end{lem}

\begin{proof}
Easy.
\end{proof}

\begin{nota}

We will have to use the following notion. {\em A sub-term $u$ of a
term $v$ comes in head position during the head reduction of $v$}.
This means the following: $v$ can be written as $C[u]$ where $C$
is some context with exactly one hole. During the head reduction
of $C$ the hole comes in head position i.e $C$ reduces to $\l
\overrightarrow{x}. ([] \ \overrightarrow{w})$ for some
$\overrightarrow{w}$. The only problem in making this definition
precise is that, during the reduction of $C$ to $\l
\overrightarrow{x}. ([] \ \overrightarrow{w})$ the potentially
free variables of $u$ may be substituted and we have to deal with
that. The notations and tools developed in, for example,
\cite{Dav01}, \cite{DN95} allows to do that precisely. Since
this is intuitively quite clear and we do not need any technical
result on this definition, we will not go further.

\end{nota}

\section{The different cases}

Let $t=\l x. F[x]$ be a closed term. First observe that

\begin{enumerate}
\item If $x \not\in \b\O(F)$, then the range of $t$ is a singleton.
  \item If $x \in \b\O(F)$ and $F$ is normalizable,  then the range of 
$t$ is trivially infinite since then the
set $\{[F[x := \l x_1...\l x_n c_k]]_{\mathcal{H}}$ / $k \in \N\}$ is
infinite where $n$ is the size of  the normal form of $F$.
  \item More generally, if $t$ has a finite B\"ohm tree then it
  satisfies the range property.
\end{enumerate}

\medskip
\noindent From now on,  we thus fix terms $F$ and $A$. We assume that:

\begin{itemize}
  \item $F$ is not normalizable and
has a unique free variable denoted as $x$ such that $x \in
\b\O(F)$. In the rest of the paper we will write $t[A]$ instead of
$t[x:=A]$.
  \item $F[A]\not \simeq F[\Omega]$
\end{itemize}

\begin{nota}\hfill
\begin{enumerate}
  \item The different occurrences of a free
  variable $x$ in a term will be denoted as $x_{[i]}$ for
  various indexes $i$.
  \item
Let $t$ be a term and $x_{[i]}$ be an occurrence of the (free)
variable $x$ in $t$. We denote by  $Arg(x_{[i]},t)$  (this
  is called the scope of $x_{[i]}$ in \cite{Bar85}) the maximal
list of arguments of $x_{[i]}$ in $t$ i.e
  the list $V$ such that
  $([] \ V)$ is the applicative context of $x_{[i]}$ in $t$.
  \end{enumerate}
\end{nota}

\noindent Since $Arg(x_{[i]},t)$ may contain variables that are bounded in
$t$ and since a term is defined modulo $\alpha$-equivalence, this
notion is not, strictly speaking, well defined. This is not
problematic and we do not try to give a more formal definition.

\begin{lem}\label{residu}
Assume $u \tr^* v$ and $x_{[i]}$ (resp. $x_{[j]}$) is an
occurrence of $x$ in $u$ (resp. $v$) such that $x_{[j]}$ is a
residue of $x_{[i]}$. Then $Arg(x_{[i]}, u) \sqsubseteq Arg(x_{[j]},v)$.
\end{lem}

\begin{proof}
Since the relation $\sqsubseteq$ is transitive it is enough to
show the result when $u \tr v$. This is easily done by considering
the position of the reduced redex.
\end{proof}

\begin{defi}
Let $x_{[i]}$ be an occurrence of $x$ in some term $t$. We say
that $x_{[i]}$ is pure in $t$ if there is no other occurrence
$x_{[j]}$ of $x$ in $t$ such that $x_{[i]}$ occurs in one of the
elements of the list $Arg(x_{[j]},t)$.
\end{defi}

For example, let $t = (x \ (x \ y))$. It can be written as  $(x_{[1]} \ (x_{[2]} \ y))$
 where the occurrence $x_{[1]}$ is pure but
$x_{[2]}$ is not pure and $Arg(x_{[1]}, t)=(x_{[2]} \ y)$.

Note that, if $x_{[1]}, ..., x_{[n]}$ are all the pure occurrences
of $x$ in $t$, there is a context $C$ with holes $[]_1, ..., []_n$
such that $t=C[[]_i=(x_{[i]} \ Arg(x_{[i]},t)) : i=1 ... n]$ and
$x$ does not occur in $C$.

\begin{lem}\label{ajout}
Let $t, t'$ be some terms such that  $t$ reduces to  $t'$. Assume
that $x_{[i']}$ is a residue in  $t'$ of  $x_{[i]}$ in $t$ and $
x_{[i']}$ is pure in $ t'$. Then  $x_{[i]}$ is pure in  $t$.
\end{lem}
\begin{proof}
Immediate.
\end{proof}

The next technical result is akin to Barendregt's lemma discussed
by de Vrijer
 in Barendregt's festschrift (see \cite{Vri07}).
It has a curious history discussed there. First proved by van Dalen
 in \cite{Daa80}, it appears as an exercise in Barendregt's book at the
end of chapter 14. Its truth may look strange. Note that the
reductions coming from the term $B$ are done {\em in the holes of
the reduct   $G$} of $F$.

\begin{lem}\label{reduction}
Let $B$ be a closed term. Assume $F[B]\tr^* t$. Then there is a
$G$ such that
\begin{itemize}
\item $F\tr^*G=D[[]_i=w_i \ : \ i \in {\mathcal I}]$ where $D$ is a context with  holes
$[]_i$ (indexed by the set ${\mathcal I}$ of all the pure occurrences $x_{[i]}$ of $x$ in $G$)
and $w_i=(x_{[i]} \ Arg(x_{[i]}, G))$.
\item  $t=D[[]_i=w'_i \ : \ i \in {\mathcal  I}]$
where $w_i[B] \tr^*w'_i$.
\end{itemize}
\end{lem}
\begin{proof}
By induction on $\langle lg(F[B]\tr^* t), cxty(F)\rangle$ where
$lg(F[B]\tr^* t)$ is  the length of a {\em standard} reduction of
$F[B]$ to $t$ and $cxty(F)$ is the complexity of $F$, i.e the
number of symbols in $F$.

- If $F = \l y.F'$, then $t = \l y.t'$ where $F'[B] \tr^*t'$.
Since $lg(F[B]\tr^* t) = lg(F'[B]\tr^* t')$ and  $cxty(F') <
cxty(F)$, we conclude  by applying the induction hypothesis on the
reduction $F'[B] \tr^*t'$.

- If $F=(y \ F_1...F_n)$, then $t = (y \ t_1...t_n)$ where $F_i[B]
\tr^* t_i$. Since $lg(F_i[B]\tr^* t_i) \leq lg(F[B]\tr^* t)$ and
$cxty(F_i) < cxty(F)$, we conclude  by applying the induction
hypothesis on the reductions $F_i[B] \tr^* t_i$.

- If $F=(\l y. U \ V \ F_1...F_n)$ where the head redex is not
reduced during the reduction $F[B]\tr^* t$, then $t = (\l y u \ v
\ f_1...f_n)$ where $U[B] \tr^*u$, $V[B] \tr^*v$ and $F_i[B]
\tr^*f_i$. Since $lg(U[B]\tr^* u) \leq lg(F[B]\tr^* t)$,
$lg(V[B]\tr^* v) \leq lg(F[B]\tr^* t)$ , $lg(F_i[B]\tr^* f_i) \leq
lg(F[B]\tr^* t)$, $cxty(U) < cxty(F)$, $cxty(V) < cxty(F)$ and
$cxty(F_i) < cxty(F)$, we conclude  by applying the induction
hypothesis on the reductions $U[B] \tr^*u$, $V[B] \tr^*v$, $F_i[B]
\tr^*f_i$.

- If $F=(\l y.U \ V \ \overrightarrow{F})$ and the first step of
the standard reduction reduces the head redex, then $F[B] \tr =
(U[y:=V] \ \overrightarrow{F})[B] \tr^* t$. Let $F' = (U[y:=V] \
\overrightarrow{F})$. Since $lg(F'[B] \tr^* t) < lg(F[B]\tr^* t)$,
we conclude  by applying the induction hypothesis  on the
reduction $F'[B] \tr^* t$.

- If $F = (x \ \overrightarrow{F})$, then $G=F$ and $D$ is the
term made of a single context $[]$ and $w=F$.
\end{proof}

\noindent The next lemma concerns the reduction of $F$ under a recursive
cofinal strategy. The canonical one is the Gross--Knuth strategy,
where one takes, at each step, the full development of the
previous one.\enlargethispage{\baselineskip}

\begin{lem}\label{Fk}
There is a sequence $(F_k)_{k \in \N}$ such that
\begin{enumerate}
  \item $F_0=F$ and, for each $k$, $F_k \tr^* F_{k+1}$.
  \item If $F \tr^* G$, then $G \tr^* F_k$ for some $k$.
  \item The function $k \hookrightarrow \#(F_k)$ is recursive.
\end{enumerate}
\end{lem}

\begin{proof}
By Theorem \ref{CR}, choose $F_{k+1}$ as a common reduct of $F_k$
and all the reducts of $F$ in less than $k$ steps.
\end{proof}

\begin{defi}
Let $x_{[i]}$ be an occurrence of $x$ in some $F_k$.  We say that
$x_{[i]}$ is good in $F_k$ if it satisfies the following
properties:
\begin{itemize}
\item $x_{[i]}$ is pure in $F_k$.

\item $u[A]$ is solvable for every sub-term $u$ of $F_k$ such that
$x_{[i]}$ occurs in $u$.

Note that this implies that $(x_{[i]} \ Arg(x_{[i]}, F_k))[A]$ is solvable.
\end{itemize}
\end{defi}

\noindent Observe that every pure occurrence of $x$ in $F_{k+1}$ is a
residue of a pure occurrence of $x$ in $F_k$. This allows the
following definition.

\begin{defi}\hfill
\begin{enumerate}
\item Let ${\mathcal T}$ be the following tree. The level $k$ in ${\mathcal T}$ is the set of
pure occurrences of $x$ in $F_k$. An occurrence  $x_{[i]}$ of $x$
in $F_{k+1}$ is the son of an occurrence $x_{[j]}$ of $x$ in $F_k$
if $x_{[i]}$ is a residue of $x_{[j]}$.

\item A branch in ${\mathcal T}$ is good if, for each $k$, the occurrence $x_{[k]}$ of
$x$ in $F_k$ chosen by the branch is good in $F_k$.
\end{enumerate}
\end{defi}

\begin{thm}\label{tree}
There is an infinite branch in ${\mathcal T}$ that is good.
\end{thm}

\begin{proof}
By Konig's Lemma it is enough to show: (1) for each $k$, there is
an occurrence of $x$ in $F_k$ that is good and (2) if the son of
an occurrence $x_{[i]}$ of $x$ is good then so is $x_{[i]}$.
\\ (1) Assume first that there is no good occurrences of $x$ in $F_k$.
This means that, for all pure occurrences $x_{[i]}$ of $x$ in
$F_k$, either $(x_{[i]} \ Arg(x_{[i]}, F_k))[A]$ is unsolvable or
this occurrence appears inside a sub-term $u$ of $F_k$ such that
$u[A]$ is unsolvable. But, if $(x_{[i]} \ Arg(x_{[i]}, F_k))[A]$
is unsolvable then so is $(x_{[i]} \ Arg(x_{[i]}, F_k))[\Omega]$
and, if $u[A]$ is unsolvable, then so is $u[\Omega]$ (proof :
consider the head reduction of $u[A]$ ; either $A$ comes in head
position or not ; in both cases the result is clear). This implies
that $F[A]\simeq F[\Omega]$. \\ (2) follows immediately from the
fact that a residue of an unsolvable term also is unsolvable and
that, if $(x \ U)[A]$ is unsolvable and $U \sqsubseteq V$ then so
is $(x \ V)[A]$.
\end{proof}

\begin{exa}
Let $G$  be a $\l$-term such that $G \tr^* \l u \l v. (v \ (G \ (x \
u) \ v))$ and $F = (G \ I)$. If we take  $ \l v. (v^k (G \ (x \ (x
\ ...(x \ I))) \ v))$ for $F_k$, what is the good
occurrence of $x$ in $F_k$ which appears in a good
branch in ${\mathcal T}$? It is none of those in $(x \ (x \ ...(x \
I)) ...)$, it is the one in $G$ !
\end{exa}\medskip

\centerline{{\bf From now on, we fix  an infinite branch in ${\mathcal
T}$ that is good}}

\medskip

\begin{nota}
We denote by $x_{(k)}$ the occurrence of $x$ in $F_k$ chosen by
the branch. Let $U_k=Arg(x_{(k)}, F_k)$.
\end{nota}

\begin{lem}\label{urs}
There is a  sequence $(\s_k)_{k \in \N}$ of substitutions and
 there are sequences $(S_k)_{k \in \N}$, $(R_k)_{k \in \N}$ of finite sequences of terms  such
that, for each $k$, $R_k$ is obtained from $\s_k(U_k)$ by some
reductions and $U_{k+1}=R_k:: S_k$.
\end{lem}
\begin{proof}
This follows immediately from Lemma \ref{residu}.
\end{proof}
\nobreak
\begin{defi}
We define the sequence $V_k$ by : $V_0 = U_0$ and  $V_{k+1} =
\s_k(V_k):: S_k$.
\end{defi}
\newpage

\begin{lem}\label{prem}\hfill
\begin{enumerate}
\item For each $k$, $V_k \tr^* U_k$.
\item For each $k' > k$, there is a substitution $\s_{k'k}$ such that $\s_{k'k}(S_k)$ is a sub-sequence of $V_{k'}$.
\end{enumerate}
\end{lem}
\begin{proof}
This follows immediately from Lemma \ref{urs}. If $k'=k+1$,
$\s_{k'k}=id$. Otherwise $\s_{k'k}=\s_{k'-1}\circ \s_{k'-2} \circ
... \circ \s_{k+1}$
\end{proof}

\begin{defi}
We define, by induction on $k$, the sequence $\rho_k$ of
  reductions and the terms $t_k$ as follows.

\begin{enumerate}
  \item
 $\rho_0$ is the head reduction of $(A \ V_0[A])$ to its head normal form.
 \item $t_k = \l \overrightarrow{z_k}. \ (y_k
\ \overrightarrow{w_k})$ is the result of $\rho_k$.
 \item $\rho_{k+1}$ is the head reduction of $(\s_k(t_k) \ S_k[A])$ to
its head normal form.
\end{enumerate}
\end{defi}

\begin{lem}\label{head}
The term  $t_k$ is the  head normal form of $(A \ V_k[A])$.
\end{lem}

\begin{proof}
Easy.
\end{proof}

\medskip

\noindent {\bf Notation and comments}

\begin{enumerate}
  \item Denote by $\r$ the infinite sequence of reductions $\r_0, \r_1, ...,\r_k,
  ...$. Note that it is not the reduction of one unique term. $\r_0$  computes  the head normal
form $t_0$ of $(A \ V_0[A])$. The role of $\s_0$ is to substitute
in the result the substitution that changes $U_0$ into the first
part of $U_1$. Note that, by Lemma \ref{urs}, this first part may
also have been reduced but here we forget this reduction. Then we
use $\r_1$ to get the head normal form $t_1$ of $(\s_0(t_0) \
S_1[A])$ and keep going like that.
  \item Note that, by Lemma \ref{prem} and \ref{head}, $t_k$  is {\em some} head
normal form for $(A \ U_k[A])$ but it is not the {\em canonical}
one i.e. the one obtained by reducing, at each step, the head
redex.
\end{enumerate}

\begin{defi}
Say that $S_k$ comes in head position during $\rho$ if, for some
$k' > k$, an element of the list $\s_{k'k}(S_k)[A]$ comes in head
position during the head reduction of $(A \ V_{k'}[A])$.
\end{defi}

\begin{defi}\label{applied}
\begin{enumerate}
  \item Let $t$ be a term with a free variable $\n$.
  \begin{enumerate}
\item
 Say that  $\nu$
is never applied in a reduct of $t$ if no reduct $t'$ of $t$
contains a sub-term of the form $(\n \ u)$.
\item Say that $\n$ is persisting in $t$ if $\n \in \b\O(t)$ and
$\nu$ is never applied in any reduct of $t$.
\end{enumerate}
  \item We say that the term $F$ has the Barendregt's persistence property if we can find a term
$A'$ that has a free variable $\n$ that is persisting in $F[A']$.
\end{enumerate}

\end{defi}

\begin{coxa}\label{CX:coxa}\hfill
\begin{enumerate}
 \item The condition ``$\n$  is never applied'' in the
previous definition implies that, letting
$A_n=A'[\n=c_n]$, a reduct of $F[A_n]$ is, essentially, a reduct
of $F[A']$.

\item Here is an example.
Let $G$  be a $\l$-term such that $G \tr^* \l u \l v. (v \ (G \ (u
\ I) \ v))$ and $F = (G \ x)$. We can take  $\l z. (z^k \ (G \ (x
\ I^{\sim k}) \ z))$ for $F_k$. It follows easily that $F[I] \not
\simeq F[\Omega]$. We  have $U_k = I^{\sim k}$, $S_k = I$, $t_k =
I$ and $\s_k = id$.

Let $J$ be a $\l$-term such that $J \tr^* \l u \l v \l y. (v \ (J
\ u \ y))$ and $I' = (J \ \nu \ I)$.

If $F[I'] \tr^* t$, then, by Lemma  \ref{reduction}, $t \tr^* \l z.
(z^n \ (G \ I'' \ z))$ for some $n$ where $I''$ is a reduct of $(I'
\ I^{\sim n})$. It is easily checked that $\nu$ is persisting in
$F[I']$. Since no $c_n$ occurs as a sub-term of a reduct of
$F[I']$, it is not difficult to show that, if $n \neq m$, then
$F[I_n] \not \simeq F[I_m]$ where $I_n = I'[\nu := c_n]$ and thus
$\Im (\l x.F)$ is infinite.
\end{enumerate}
\end{coxa}

\newpage
\begin{thm}\label{cases}\hfill
\begin{enumerate}

\item Assume first that the length of the $U_k$ are bounded. Then,
$F$ has the Barendregt's persistence property.

\item Assume next that the length of the $U_k$ are not bounded and the
branch we have chosen  in ${\mathcal T}$  is {\bf recursive}.

\begin{enumerate}

\item If the set of those $k$ such that $S_k$ comes in head position during
$\rho$ is infinite, then $F$ has the Barendregt's persistence property.

\item Otherwise it is possible that $F$ does not have the Barendregt's 
persistence pro-perty.

\end{enumerate}

\end{enumerate}
\end{thm}

\begin{proof}\hfill
\begin{enumerate}

\item It follows immediately from Lemma \ref{residu} that there are $ l,k_0 > 0$
such that for all $k \geq k_0, lg(U_k) = l$. The fact that $F$ has the
Barendregt's persistence property is proved in section 3.

\item

\begin{enumerate}

\item The fact that $F$ has the Barendregt's persistence property is proved in section
4.

\item There are actually two cases and the reasons why we cannot
find a term $A'$ are quite different.
\begin{enumerate}

\item For all $k$ there is $k' >  k$ such that $y_{k'} \in
dom(\s_{k'})$. The fact that the head variable of $t_k$ may change
infinitely often does not allow to use the technic of sections 3
or 4. A. Polonsky has given a term $F$ that corresponds to this
situation and such that a variable $\n$ can never be persisting in
a term $A'$ of the form $\l x_1 ... x_n. (x_i \ w_1 ... \ w_m)$.
See example 1 below.

\item For some $k_1$, $y_{k} \not \in dom(\s_{k})$ for all $k \geq k_1$.
Since there are infinitely many $k \geq k_1$ such that $S_k$ is
non empty this implies that, after some steps, $t_k$ does not
begin by $\l$. Thus, there is $k_2$ and $y$, such that, for all $k
\geq k_2$, $t_k = (y \ \overrightarrow{w_k})$. Using the technic
of sections 3 or 4 allows to put a term $J$ in front of some
(fixed) element of the sequence $\overrightarrow{w_k}$ but this is
not enough to keep $\n$. We adapt the example of A. Polonsky to
give a term $F$ that corresponds to this situation and such that a
variable $\n$ can never be persisting in a term $A'$ of the form
$\l x_1 ... \l x_n. (A \ w_1 ... \ w_n)$ where $w_j \simeq \l y_1
... \l y_{r_j}. (x_j \ w^j_1 ... w^j_{r_j})$. See example 2 below.\qedhere
\end{enumerate}
\end{enumerate}
\end{enumerate}
\end{proof}

\begin{coms}\label{CM:first}\hfill
\begin{enumerate}
  \item  For  case 1. we will  give two proofs. The first
  one
  is quite simple. The second one is much more elaborate and even though it,
  actually, does not work for all the possible situations, we give
  it because it is an introduction to the more complex
  section 4.
In the first proof,  we simply use
  the fact that the length of the $U_k$ are bounded to find a term $A'$ that has nothing to do with $A$.
  In the second proof and in section 4, the term $A'$ that we give
  has the Bargendregt's property and
  behaves like $A$ (using the idea of \cite{Bar08}) in the sense that it looks
 like an infinite $\eta$-expansion of $A$.

\item  When we say, in case 2.(b) of
the theorem,  that it is possible that $F$ does not have the
Barendregt's persistence property we are a bit cheating. We only show (except
in example 1) that there is no $A'$ with a persisting $\n$
satisfying an extra condition. This condition is that $A'$ looks
like $A$ i.e. the first levels of the B\"{o}hm tree of $A'$ must
be, up-to some $\eta$-equivalence, the same as the ones of $A$.

\item  It is known  that
there are recursive (by this we mean that we can compute their
levels) and infinite trees such that each level is finite and that have
no recursive infinite branch. We have not tried to transform such
a tree in a lambda term such that the corresponding ${\mathcal T}$ has
no branch that is good and recursive but we guess this is
possible.
\end{enumerate}
\end{coms}

\begin{exa}\label{E:one}
This example is due to A. Polonsky. Let $G,H$  be $\l$-terms such
that \\$G \tr^* \l y \l z. (z \ G \ \l u. (y \ (K \ u)) \ z)$ and $H
\tr^* \l u \l v \l w. (w \ (H \ u \ (v \ u) \ w))$. \\Let $F = (G \ \l y. (H \ y \ x))$. \\
We have $F \tr^* \l z. (z^n \ (G \  \l y\l w.
(w^m  \ (H \  (K^n \ y) \ (x \ (K^n \ y)^{\sim m}) \ w)) \ z))$. \\
Thus, if $B \simeq \l y_1 ... \l y_r. (y_i \ w_1 ... w_l)$ where
$\nu$ is possibly free in the $w_j$, \\
$F[B] \simeq \l z. (z^l \ (G
\ \l y\l w. (w^r  \ (H \ (K^l \ y)\ y \ w)) \ z))$ and  $\nu$ is
not persisting in $F[B]$ (since it can be erased). This means that
for all closed term $A$ such that $F[A] \not \simeq F[\Omega]$,
and for all solvable term $A'$, $\nu$ is not persisting in $F[A']$.
\end{exa}

\begin{exa}\label{E:two}
This example is an adaptation of the previous one. Let $G,H$  be
$\l$-terms such that $G \tr^* \l y \l z. (z \ (G \ \l u. (y \ (K \
u)) \ z))$ and $H \tr^* \l u \l v \l w. (w \ (H \ u \ (v \ u) \ w))$.\\
Let $F = \l y.(G \ \l v. (H \ v \ (x \ y)))$. \\ We have $F \tr^*
\l y\l z. (z^n \ (G \  \l v\l w. (w^m \ (H \ (K^n \ v) \ (x \ y \ (K^n \
v)^{\sim m}) \ w)) \ z))$ and it is clear that $F[I] \not \simeq
F[\Omega]$.
\\
Let $I'  \simeq \l x \l x_1 ... \l x_r. (x \ w_1 ... w_r)$ where,
for $1 \leq j \leq r$, $w_j \simeq \l y_1 ... \l y_{r_j} (x_j \
w^j_1 ... w^j_{r_j})$. \\
For $n \geq max_{1\leq j \leq r}(r_j)$,
$F[I'] \simeq \l y\l z. (z^n \ (G \ \l v\l w. (w^m \ (H \ (K^n \
v) \ (y \ w'_1 ... w'_r) \ w)) \ z))$ for some $w'_j$ where $\nu$
does not occur and thus $\nu$ is not persisting in $F[I']$.

Note, however,  that  $\nu$ is persisting in $F[I'']$ where $I'' =
\l z. (z \ \nu)$.
\end{exa}

\section{Case 1 of Theorem \ref{cases}}\label{cas1}

We assume in this section that we are in case 1. of Theorem
\ref{cases}.

\subsection{A simple argument}\hspace*{\fill} \\

\noindent Let $A' =  \l x_1...\l x_l \l z. (z \ \nu)$ where $l$ is a bound
for the length of the $U_k$. We show that $\n \in \b\O(F[A'])$.
 It is easy to show that $\n$ is never applied  in the
terms $(A' \ U_k[A'])$ but the fact that $\n$ is never applied in
a reduct of $F[A']$ is not so clear. Since, because of the next
section, we do not need this point we have not tried to check.

If  $F[A'] \tr^* H \tr_{\O}^* G$. By Lemma \ref{reduction}, $F
\tr^*F'=D[[]_i=w_i \ : \ i \in {\mathcal I}]$, $w_i=(x_{(i)} \ Arg(x_{(i)}, F'))$ and
$H=D[[]_i=w'_i \ : \ i \in {\mathcal  I}]$ where $w_i[A'] \tr^*w'_i$. Let $k$
be such that $F' \tr F_k$. Let $i_0 \in {\mathcal  I}$ be such that the
occurrence of $x_{(k)}$ in $F_k$ chosen by the good branch is a
residue of the occurrence of $x_{[{i_0}]}$ in $F'$. By Lemma
\ref{residu}, the length of $Arg(x_{[{i_0}]}, F')$ is bounded by
$l$ and thus $w'_{i_0} = \l x_q ... x_l\l z. (z \
 \nu)$ for some $q \leq l$.

It remains to show that the sub-term  $\l z. (z \
 \nu)$ of $w'_{i_0}$ cannot be erased in the $\O$-reduction from $H$ to $G$.
 Assume it is not the case. Then,
 there is a
sub-term $D'$ of $D$ containing the hole $[]_{i_0}$ such that
$D'[[]_i=w'_i \ : \ i \in {\mathcal  I}]$ is unsolvable. $D'[[]_i =
w_i]$ is solvable (since, otherwise, the occurrence $x_{[k]}$ will
be in an unsolvable sub-term of $F_k$ and this contradicts the
fact that  $x_{[k]}$ is  good). Since the reduction $D'[[]_i =
w_i] \tr^* D'[[]_i =  w'_i]$ only is inside the $w'_i$, since the
first term is solvable and the second one is not, then, by the
Church-Rosser property, the head variable of the head normal  form
of  $D'[[]_i = w_i]$ is an occurrence of $x$. By Lemma
\ref{ajout},
 $x_{[i_0]}$ is  pure in $F'$.
$x_{[i_0]}$ is not the head  variable of the head normal form of
$D'[[]_i = w_i]$  (since, otherwise, by the  Church-Rosser
property, $D'[[]_i = w'_i]$ would be solvable). Contradiction.

\subsection{The proof}\label{4.2}\hspace*{\fill} \\

\noindent There are actually different situations.
\begin{enumerate}

\item Either, for some $ k_1 \geq k_0$,  $y_k \in \overrightarrow{z_k}$ for any $k \geq
k_1$. Since for $k \geq k_1$, $S_k$ is empty and the head variable
of $t_k$ is not  substituted,
 there is $\overrightarrow{z}$ and $z \in \overrightarrow{z}$,
such that, for $k \geq k_1$, $t_k = \l \overrightarrow{z}. \ (z \
\overrightarrow{w_k})$.

\item Or $y_k \not \in \overrightarrow{z_k}$ for all $k \geq
k_0$ and

\begin{enumerate}

\item Either the situation is unstable (i.e.  the set of those $k$  such that $y_{k}
\in dom(\s_{k})$ is infinite).

\item Or, there is
$\overrightarrow{z}$, $y \not \in \overrightarrow{z}$ and $k_1
\geq k_0$, such that, for $k \geq k_1$, $t_k = \l
\overrightarrow{z}. \ (y \ \overrightarrow{w_k})$.

\end{enumerate}

\end{enumerate}

\noindent We assume  that we are in situation 1. or 2.(b) which may be
synthesized by: there exists $k_0$ and some fixed variable $y$
(that may be in $\overrightarrow{z}$ or not), such that, for all
$k \geq k_0$, $t_k=\l \overrightarrow{z}. \ (y \
\overrightarrow{w_k})$ for some $\overrightarrow{w_k}$.

We fix $p \geq
lg(\overrightarrow{z})+ l + 2$ where $l$ is a bound for  the
length of the $U_k$. Let $A' = (J \ \n \ A)$ where $J$ is a new
constant
 with the following reduction rule
\[(J \ \n \ u) \tr \l y_1 ... \l y_p. (u \ (J \ \n \
y_1) ... (J \ \n \ y_p))\]
We will prove that $\n$ is persisting in $F[A']$.

\medskip

%\noindent {\em Note}

The term   $A'$ is not a pure $\l$-term since the constant $J$
occurs in it. We could, of course, replace this constant by  a
$\l$-term $J'$ that has the same behavior, e.g. \\$(Y \; \l k \l y_1
... \l y_p. (u \ (k \ \n \ y_1) ... (k \ \n \ y_p)))$ where $Y$ is
the Turing fixed point operator. But such a term introduces some
problems in Lemma \ref{commut-2} because $J'$ and $(J' \ \n)$
contain redexes and can be reduced. With such a term $J'$, though
intuitively true, this lemma (as it is stated) does not remain
correct.

Making this lemma correct (with $J'$  instead of a constant) will
require the treatment of  redexes  inside $J'$ and $(J' \ \n)$.
The reader should be convinced that this can be done but, since it
would need tedious  definitions, we will not do it.

\medskip

Note that, in  situation 2.(a),  we cannot do the kind of proof
given below.  Here is an example. Let $G$  be a $\l$-term such
that $G \tr^* \l u \l v. (v \ (G \ \l y. (u \ (K \ y)) \ v))$ and
$F = (G \ x)$. We can take $\l z. (z^k \ (G \ \l y. (x \ (K^k \
y)) \ z)$ for $F_k$. We thus have $U_k = (K^k \ y)$, $t_k = \l
z_1...\l z_k.\ y$ and $\s_k = [y := (K \ y)]$. It is clear that
$F[I] \not \simeq F[\Omega]$.
But $F[I'] \tr^*  \l z. (z^k
\ (G \ I \ z))$ for any $I' \simeq \l y \l y_1 ... \l y_k. \ (y \
w_1 ... w_k)$
%for any $I' \simeq \l y \l y_1 ... \l y_m. \ (y \
%w_1 ... w_m)$
where $\nu$ is possibly free in the $w_i$. Thus
$\nu$ is not persisting in $F[I']$.

\subsubsection{Some preliminary definitions and results}\hspace*{\fill} \\

\noindent When, in a term $t$, we replace some sub-term $u$ by $(J \ \nu \
u)$ to get $t'$, the reducts $u$ (resp. $u'$), of $t$ (resp. of
$t'$) are very similar. The goal of this section is to make this a
bit precise.

\begin{defi}\hfill
\begin{enumerate}
  \item We define, for terms $u$, the sets $E_u$ of terms by the following
grammar:\\ $E_u = u \ | (J \ \nu \ e_u) \ | \l \overrightarrow{y}.
(e_u \ \overrightarrow{e_y})$
  \item Let  $t,t'$ be some terms. We denote by $t \leadsto t'$
  if there is a context $C$ with one hole such that $t=C[u]$ and
  $t'=C[u']$ where $u' \in E_u$.
\end{enumerate}
\end{defi}

\begin{noco}\label{NC:noco}\hfill
\begin{enumerate}
\item Note that, in the previous definition as well as in the sequel,
$e_u$ always denotes a term in $E_u$
and, for a sequence $\overrightarrow{y}$ of variables ,
$\overrightarrow{e_y}$ always denote a sequence of terms
$\overrightarrow{v}$ (of the same length as $\overrightarrow{y}$)
such that for each variable $y$ in $\overrightarrow{y}$, the
corresponding term in $\overrightarrow{v}$ is a member of $E_y$.
Also note that,  in the previous definition,  for a term $\l
\overrightarrow{y}. (e_u \ \overrightarrow{e_y})$ to be in $E_u$,
we assume that the variables in $\overrightarrow{y}$ do not occur
in $e_u$.
 \item  $t \leadsto t'$ means that $t'$ is obtained from $t$ by
replacing some sub-term $u$ of $t$ by $(J \ \n \ u)$ or by
reducing  redexes introduced by $J$ i.e.  the one coming from its
reduction rule  $(J \ \n \ u) \tr \l y_1 ... \l y_p. (u \ (J \ \n
\ y_1) ... (J \ \n \ y_p))$  and those whose $\l$'s are among $\l
y_1 ... \l y_p$.
\end{enumerate}
\end{noco}

\begin{lem}\label{utile}\hfill
\begin{enumerate}
\item If $t \in E_z$ and $t' \in E_u$ then $t[z:=t'] \in E_u$.
\item If  $a  \leadsto a'$ and $b  \leadsto b'$, then $a[x:=b] \leadsto^* a'[x:=b']$.
\item If $u \leadsto^* u'$ and if $v'$ is an $\O$-redex in $u'$, then $v \leadsto^* v'$
for some $\O$-redex  $v$ in $u$.
\item If $t \in E_u$ then $t \tr^* \l \overrightarrow{y} (u \
\overrightarrow{e_y})$ for some $\overrightarrow{e_y}$.

\end{enumerate}
\end{lem}

\begin{proof}
1, 2 and 3 are immediate. 4 is proved by induction on the number
of rules used to show $t \in E_u$.
\end{proof}

\begin{lem}\label{commut-1}\hfill
\begin{enumerate}
\item Assume $u =C[(\l x. a \ b)]$ for some context $C$ and let $u
\leadsto^*   u'$. Then $u'=C'[(a' \ b')]$ for some context
$C'$ and some terms $a',b'$ such that $C \leadsto^* C'$,  $a' \in   E_{\l x.a''}$,
$a   \leadsto^*   a''$ and  $b \leadsto^* b'$.
\item Assume $u \leadsto^* u'$ and $u \tr^* v$. Then,  $v
\leadsto^* v'$ for some $v'$ such that $u' \tr^* v'$.
\end{enumerate}
\end{lem}

\begin{proof}
The first point is immediate because the operations that are done to go from $u$ to $u'$ are either in
$C$ (to get $C'$) or in $b$ (to get $b'$) or  in $a$ (to get $a'$)
or in $\l x.a'$. It is enough to check that locals and globals
operations on $a$ commute.

For the second point, we do the proof for one step of reduction $u \tr v$ and we use the first point and Lemma
\ref{utile}.
\end{proof}

\begin{lem}\label{commut-11}\hfill
\begin{enumerate}
\item If $\overrightarrow{\l x}. (x \ \overrightarrow{c}) \leadsto^* u$,
then $u \tr^* \overrightarrow{\l x} \overrightarrow{\l y}.(x \
\overrightarrow{c'} \ \overrightarrow{e_y})$ for some
$\overrightarrow{c}\leadsto^* \overrightarrow{c'}$ and some
$\overrightarrow{e_y}$.
\item If $u \tr^* \overrightarrow{\l x}. (x \ \overrightarrow{c})$ and $u \leadsto^* u'$,
then $u' \tr^* \overrightarrow{\l x} \overrightarrow{\l y}.(x \
\overrightarrow{c'} \ \overrightarrow{e_y})$ for some
$\overrightarrow{c}\leadsto^* \overrightarrow{c'}$ and some
$\overrightarrow{e_y}$.
\end{enumerate}
\end{lem}

\begin{proof}\hfill
\begin{enumerate}
  \item We do the proof on an
example. It is clear that this is quite general. Assume $\l x_1 \l
x_2 (x \ c_1 \ c_2) \leadsto^* u$, then $u \in E_{\l x_1 u_1}$,
$u_1 \in E_{\l x_2 u_2}$, $u_2 \in E_{(v_1 \ c'_2)}$, $v_1 \in
E_{(v_2 \ c'_1)}$, $v_2 \in E_{x}$, $c_1 \leadsto^* c'_1$ and $c_2
\leadsto^* c'_2$. The $\b$-reduction  is done starting from inside
and  the propagation is done by using Lemma \ref{utile}.
  \item It is a consequence of the first point and Lemma \ref{commut-1}.\qedhere
\end{enumerate}
\end{proof}

\noindent The next lemma means that, when a redex appears in some $u'$ where
$u \ \leadsto^* \ u'$, it can either come from the corresponding
redex in $u$, or has been created by the transformation of an
application in $u$ that was not already a redex or comes from the
replacement of  some sub-term $u$ by, essentially, $(J \ \nu \
u)$.
\begin{lem}\label{commut-2}\hfill
\begin{enumerate}
\item Assume $u' =C'[R']$ for some context $C'$ and some redex $R'$ and let $u
\ \leadsto^*  \ u'$. Then :
\begin{itemize}
\item either $R' = (\l x. a' \ b')$, $u=C[(\l x. a \ b)]$ for some context
$C$ and some terms $a,b$ such that $C \leadsto^*  C'$,  $a \leadsto^* a'$ and
$b \leadsto^* b'$.
\item or $R' = (a' \ b')$, $u=C[(a \ b)]$ for some context
$C$ and some terms $a,b$ such that $C \leadsto^*  C'$,  $a' = \l \overrightarrow{y}.
(a'' \ \overrightarrow{e_y})$,  $a \leadsto^* a''$ and
$b \leadsto^* b'$.
\item or $R' = (J \ \nu \ a')$,  $u=C[a]$ for some context
$C$ and some term $a$ such that $C \leadsto^*  C'$ and $a \leadsto^* a'$.
\end{itemize}
\item Assume $u \leadsto^* u'$ and $u' \tr^* v'$. Then  $v \leadsto^* v'$
for some $v$ such that $u \tr^* v$.
\end{enumerate}
\end{lem}

\begin{proof}
For the first point, there are two cases. Either the redex $R'$ is the residue of a
redex in $u$ or it has been created by the operations from the
grammar $E$.

For the second point, it is enough
to prove the result for one step of reduction $u' \tr v'$. Use the
first point.
\end{proof}

\begin{defi}
Let $t$ be a solvable term. We say that:
\begin{enumerate}
  \item  $\n$ occurs nicely in
$t$ if the only occurrences of $\n$ are in a sub-term of the form
  $(J \ \n)$.
  \item $\n$ occurs correctly in
$t$ if it occurs nicely in $t$ and the head normal form of $t$
looks like $\overrightarrow{\l x}. (x \ \overrightarrow{c} \
\overrightarrow{e_y})$ for some final subsequence
$\overrightarrow{y}$ of  $\overrightarrow{x} $ of length at least
1 such that $\n$ does occur in $\overrightarrow{e_y}$.
\end{enumerate}

\end{defi}

\begin{lem}\label{reste}
Let $t$ be a solvable term. Assume that $\n$ occurs nicely (resp.
correctly) in $t$. Then $\n$ occurs nicely (resp. correctly)  in
every reduct of $t$.
\end{lem}

\begin{proof}
 By the properties of $J$.
\end{proof}

\begin{lem}\label{pas_applique}
The variable $\n$ is never applied in a reduct of $F[A']$.
\end{lem}

\begin{proof}
The variable $\nu$ occurs nicely in $F[A']$, then, by Lemma
\ref{reste}, it occurs nicely in every reduct of $F[A']$, thus
$\n$ is never applied in a reduct of $F[A']$.
\end{proof}

\subsubsection{End of the proof}\hspace*{\fill} \\

\begin{prop}\label{p1}
For $k\geq k_0$, $\n$ occurs correctly in $(A' \ V_k[A])$.
\end{prop}

\begin{proof}
$(A \ V_k[A]) \leadsto^* (A \ (J \ \n \ V_k[A]))$ (note that this
last term may be misunderstood: it actually means $ (A \ (J \ \n
\ a_1) \ ... \ (J \ \n \ a_q))$ where $V_k[A]$ is the sequence $
a_1... a_q$) and $(A \ V_k[A]) \tr_h^* \l\overrightarrow{z}.(y \
\overrightarrow{w_k})$. Thus, by Lemma \ref{commut-11}, $(A \ (J \
\n \ V_k[A])) \tr^* \overrightarrow{\l z}\overrightarrow{\l x}. (y
\ \overrightarrow{w_k'}\ \overrightarrow{e_x})$ for some
$\overrightarrow{w_k} \leadsto^* \overrightarrow{w_k'}$ and
$\overrightarrow{e_x}$.

But $(A' \ V_k[A]) \tr^* \l \overrightarrow{y} (A \ (J \ \n \
V_k[A]) \ \overrightarrow{(J \ \n \ y)})$ and thus $(A' \ V_k[A])
\tr^* \l \overrightarrow{y}.(\l \overrightarrow{z} \l
\overrightarrow{x}.$ $ (y \ \overrightarrow{w_k'}\
\overrightarrow{w_x}) \ \overrightarrow{(J\ \n \ y)})$. Using then
 $p - l \geq lg(\overrightarrow{z})+2$ and distinguishing $y \not \in \overrightarrow{z}$ or $y \in
\overrightarrow{z}$ it follows easily that $(A' \ V_k[A]) \tr^*
\l \overrightarrow{y_1} (y_2 \ \overrightarrow{w_k''}\
\overrightarrow{w_{y_3}})$ where $\overrightarrow{y_3}$ is a final subsequence of
$\overrightarrow{y_1}$, $\overrightarrow{e_{y_3}} \in
\overrightarrow{E_{y_3}} $ and $lg(\overrightarrow{y_3})\geq 1$.

The fact that $\n$ occurs nicely is clear.
\end{proof}

\begin{prop}\label{p2}
Assume that there is  a sequence  $(j_k)_{k \in \N}$ of integers such that,
for each $k$, $\n$ occurs correctly in $(A' \ U_{j_k}[A])$. Then
$\n$ is persisting in $F[A']$.
\end{prop}
 \begin{proof}
 Let $t_1$ be a reduct of $F[A']$. By Lemma \ref{pas_applique}, it is enough to show that $\n$
 does occur in $t_1$.
 By Lemma \ref{commutation}, let $t_2$ be such that $F[A'] \tr^* t_2
 \tr^*_{\O} t_1$. By Lemma \ref{reduction},
  $F\tr^*G=D[[]_i=w_i \ : \ i \in {\mathcal I}]$ where  $w_i=(x_{[i]} \ Arg(x_{[i]}, G))$
 and  $t_2=D[[]_i=w'_i \ : \ i \in {\mathcal I}]$ where $w'_i$ is a
 $\beta$-reduct of $w_i[A']$.

 By Lemma \ref{Fk}, $G \tr^* F_{j_k}$ for some $k$. Let $x_{[j]}$
 be the occurence of $x$ in $G$ which has $x_{( {j_k})}$ as residue
 in $F_{j_k}$. By Lemma \ref{residu}, there is a substitution
 $\s$ and a sequence $V$ of terms such that
 $U_{j_k} = W :: V$ and $\s(M_j)$ reduces to $W$ where
 $M_j=Arg(x_{[j]},G)$. By Lemma \ref{commut-2}, let $w''_j$ be such that
 $(A' \ \ Arg(x_{[j]}, G)[A])$ reduces to $w''_j$  and $w''_j
 \leadsto^* w'_j$.

 Since $(A'\  \s(M_j)[A] )$ reduces both to $\s(w''_j)$ and to
 $(A'\  W[A] )$, let $s$ be a common reduct of $( \s(w''_j) \
 V[A])$ and $(A'\  W[A] \ V[A] )$. Since $\n$ occurs correctly in
 $(A'\ W[A] \ V[A] )$, it  occurs correctly (by Lemma \ref{reste})
 in $s$. But $(\s(w''_j) \ V[A])$
reduces to $s$ and $\n$ does not occur in  $\s$ neither in $V[A]$, then
$\n$ occurs in $w''_j$.  Thus it occurs in $w'_j$ and thus in $t_2$.

Since  an $\O$-reduction of $w''_j$ cannot
erase $\n$ (otherwise, by Church-Rosser, $\n$ will not occur in a
reduct of $s$) and $w''_j \leadsto^* w'_j$, then, by  Lemma \ref{utile}(item 3),
an $\O$-reduction of $w'_j$ cannot erase all its  $\nu$.
The same proof as the one in section 3.1
shows that $\n$ cannot be totally erased by the $\O$-reduction
from  $t_2$ to $t_1$ and thus it occurs in $t_1$.
 \end{proof}

\begin{cor}
$\n$ is persisting in $F[A']$.
\end{cor}
\begin{proof}
By Proposition \ref{p1}, for $k\geq k_0$, $\n$ occurs correctly in
$(A' \ V_k[A])$.  By Lemma \ref{reste},  $\n$ occurs correctly in
$(A' \ U_k[A])$ and we conclude by Proposition \ref{p2}.
\end{proof}

\section{Case 2(a) of Theorem \ref{cases}}\label{cas_2}

Here is an example of this situation.
Let $F = (G \ x)$ where $G$  is a $\l$-term such that $G \tr^* \l u
\l v. (v \ (G \ (u \ K \ K) \ v))$. We can take $\l z. (z^k \ (G \
(x \ K^{\sim 2k}) \ z)$ for $F_k$. Thus  $U_k = K^{\sim 2k}$, $S_k
= K^{\sim 2}$, $t_k = K$ and $\s_k = id$. $F[K] \not \simeq
F[\Omega]$. We are in situation 2.(a) of Theorem \ref{cases}
because $(\s_k(t_k) \ S_k[K])= (K \ K^{\sim 2}) \tr^* K$.

\medskip

\noindent {\bf The idea of the construction}

The desired $A'$ looks like the one for the previous case. It will
be $(\widehat{J}\ c_0 \ \n \ A)$ for {\em some other term}
$\widehat{J}$ that behaves mainly as $J$  but has to be a bit more
clever.  The difference with the previous one is the following.
For the $J$ of section 3.2, $(J \ \n \ u)$ reduces to a term where
the $J$ that occurs in the arguments of $u$ is {\em the same}.
Here $\widehat{J}$ has to be parameterized by some integers. This
is the role of its first argument. We will define a term
$\widehat{J}$ and will denote $ (\widehat{J} \ c_n)$ by $J_n $
i.e. intuitively, the $\widehat{J}$ at ``step'' $n$. Here $(J_n \
\n \ u)$ will reduce to a term where $J_{n+1}$ and not $J_n $
occurs in the arguments of $u$. See Definition \ref{J2} and Lemma
\ref{nu2} below.

The reason is the following. In section \ref{4.2}, after some
steps, the head variable of $(A \ V_k[A])$ does not change anymore
and the variable $\n$ can no more disappear. Here, to be able to
ensure that  $\n$ does not disappear, $\widehat{J}$ must occur in
head position infinitely often and, for that, it has to be able to
introduce more and more $\l$'s.

The $\l$'s that are introduced by $J_n$  (i.e.  $\l y_1 ... \l
y_{h_n}$ of Lemma \ref{nu2}) are used for two things. First they
ensure that a $\widehat{J}$ will be added to the next $S_k$, say
$S_{k_n}$, coming in head position. Secondly, they ensure that
$\n$ occurs correctly in $(A' \ V_{j_n}[A] )$ where $j_n$ is large
enough to let $S_{k_n}$ come in head position.

 This is the idea but there is one  difficulty. If $S_{k_n}$ comes in head
 position during the head reduction of $(A \ V_{j_n}[A] )$, it
 will also come in head position during the reduction of $(A' \ V_{j_n}[A]
 )$ but (see Lemma \ref{commut-11}) some $\l$'s are put in front (the $\overrightarrow{\l y}$ of Lemma \ref{hnfJ22}).
  If we could compute their number it will not be problematic
 but, actually, this is not possible. For the following reason:
 these $\l$'s
 come from the $J_k$ that came in head positions in previous
 steps but, when we are trying to put a $\widehat{J}$ in front of  $S_{k_n}$  it is possible that
  $J_p$ has already appeared  in head position for some $p > n$.

  We thus proceed as follows. $J_0$ introduces enough $\l$'s to put $J_1$ in
  front of $S_{k_0}$ (the first $S_k$ that comes in head position during
  $\r$) and to ensure that $\n$ will occur correctly in $(A' \ V_{j_0}[A]
 )$. Then we look at the first term  of $V_{j_0}$ that comes in
 head position during $\r$. This term, say $a_0$,  has $J_1$ in front of it in the head reduction of $(A' \ V_{j_0}[A]
 )$.
 This is the term ($a_0$ may be a term in $S_{k_0}$ but it may be some other term) that will
 allow
 to define the number of $\l's$ that $J_1$ introduces. It  introduces enough $\l's$ to put $J_2$ in front of $S_{k_1}$, the
 next
 $S_k$ coming in head position, and to ensure that $\n$ occurs correctly in the reduction of $(A' \ V_{j_1}[A]
 )$. Note that $k_1$ has to be large enough to avoid the $\l$'s that
 occur at the beginning of the term where $a_0$ occurs in head
 position. We keep going in this way to define the $J_n$. It is clear
 (this could be formally proved by using standard fixed point theorems)
 that the function
 $h$ (see definitions \ref{def_5.1} and  \ref{J2} below) that computes the parameter $c_n$ at step $n$
 is recursive.
A formalization of this is given in Definition \ref{def_5.1}
below.

 There is a final, though
 not essential, difficulty to make this precise. The definition of $\widehat{J}$ needs to know the function $h$.
 But to compute $h$ we need some approximation of $\widehat{J}$. More precisely to compute $h(n)$
 we need to know the behavior of $\widehat{J}$ where only the values of $h(p)$ for $p<n$ will be used. To do that,
 we introduce  fake $J_n$
 (they are denoted as $\tilde{J}_n$ below). They are as $J_n$ but the term
 $J_{n+1}$ (which is not yet known since $h(n+1)$ is not yet known) is replaced by some fresh
 constant $\g$.
Note that, using the standard fixed point theorem, we could avoid
these fake $\tilde{J}_n$ but then, Definition \ref{def_5.1} below
should be more complicated because the constant $\gamma$ would be
replaced by a term computed from a code of $\widehat{J}$ and we
should explain how this is computed ...

Finally note that it is at this point that we use the fact that
the branch in ${\mathcal T}$ is recursive.

 \medskip

 \noindent {\bf Back to the example}

 With the example given before, we can take for $\widehat{J}$ a $\l$-term such that\\
$\widehat{J} \tr^* \l u \l v \l y_1 \l y_2. (v \ (\widehat{J} \ u
\ y_1) \ \ (\widehat{J} \ u \ y_2))$ and for $K'$ the term
$(\widehat{J} \ \nu \ K)$. Then $K' \tr^* K'' = \l y_1 \l y_2 \l
z_1 \l z_2.(y_1 \ (\widehat{J} \ u \ z_1) \ (\widehat{J} \ u \
z_2))$. Since, if $F[K'] \tr^* t$, then $t \tr^* \l z. (z^k \ (G \
K'' \ z))$. It follows that $\nu$ is persisting in $F[K']$. Note
that here the function $h$ is constant. Given any recursive
function $h'$ it will not be difficult to build terms $F,A$ such
that the corresponding function $h$ is precisely $h'$.

\begin{defi}\label{def_5.1}
For each $i$, let $l_i$ be the number of $\l$'s at the  head of
$t_{i}$.
\end{defi}

The definition of $\widehat{J}$  needs some new objects. Let
$\gamma$ be a new constant. We define, by induction, the integers
$j_{n}$, $h_{n}$, the terms $\tilde{J}_{n}, a_{j_n}$ and the
sequence of terms  $B_{j_n}$.

In this definition a term (or a sequence of terms) marked with
$'$ is a term in relation $\rightsquigarrow$ with the
corresponding unmarked term.

\begin{itemize}
  \item ({\em Step 0} )
Let $j_0$ be the least integer $j$ such that some $S_k$  comes in
head position during $\r_j$ (the head reduction of $(A \
V_j[A])$), $k_0 = lg(V_{j_0})$, $h_0 = max (l_0, k_0)$. Let
$V_{j_0} = d_1...d_{k_0}$ and  $\tilde{J}_0 = \l n \l x \l
x_1...\l x_{h_0} (x \ (\gamma \ n \ x_1) \ ... \ (\gamma \ n \
x_{h_0}))$. Then
\[((\tilde{J}_0 \ \nu \ A) \  V_{j_0}[A]) \tr^* \l y_1...y_{r_0} (A
\ (\gamma \ \nu \ d_1[A]) \ ... \ (\gamma \ \nu \ d_{k_0}[A]) \
(\gamma \ \nu \ y_1) \ ... \ (\gamma \ \nu \ y_{r_0}).\]
Since, for some   $i$,  $d_i[A]$ comes in head position during the
head reduction of $(A \ V_{j_0}[A])$, then, for some $i'$,
$(\gamma \ \nu \ d_{i'}[A])$ comes in head position during the
head reduction of $((\tilde{J}_0 \ \nu \ A)\  V_{j_0}[A])$ and
thus $((\tilde{J}_0 \ \nu \ A)\  V_{j_0}[A]) \tr^*  \l X_{j_0}
((\gamma \ \nu \ a_{j_0}) \ B_{j_0})$ for some $a_{j_0}$, some
sequence $X_{j_0}$ of variables and some sequence $B_{j_0}$ of
terms.

 \item ({\em Step 1} )
Let $j_1$ be the least integer $j > j_0$ such that some $S_k$  comes
in head position during $\r_j$ for $k$ large enough ($lg(X_{j_0})
< lg(S_{j_0}:: ... :: S_{k-1})$ is needed). Then  
\[((\tilde{J}_0 \
\nu \ A) \ V_{j_1}[A]) = ((\tilde{J}_0 \ \nu \ A) \ V'_{j_0}[A] \
S'_{j_0}...S'_k ... S'_{j_1-1}) \tr^*  ((\gamma \ \nu \ a'_{j_0})
\ B'_{j_0}) \ T_0.\]
Let $h_1 = l_{j_0} + lg(B'_{j_0}:: T_0) + 1$ and $\tilde{J}_0 =
\tilde{J}_1 [\gamma := \l n \l x \l x_1...\l x_{h_1} (x  \ (\gamma
\ n \ x_1) \ ...$ $ \ (\gamma \ n \ x_{h_1}))]$. Then
$((\tilde{J}_1 \ \nu \ A) \ V_{j_1}[A]) \tr^*  \l X_{j_1} ((\gamma
\ \nu \ a_{j_1}) \ B_{j_1})$ for some term $a_{j_1}$ and  some
sequence $B_{j_1}$ of terms.

  \item ({\em Step n+1} )
Assume the integers $j_{n}$, $h_{n}$ and the term $\tilde{J}_{n}$
are already defined.  Let $j_{n+1}$ be the least integer $j > j_n$
such that some $S_k$ comes in head position during $\r_j$ for $k$
large enough ($lg(X_{j_n}) < lg(S_{j_1}:: \ ...\ :: \ S_{k-1})$ is
needed). Then 
\[((\tilde{J}_n \ \nu \ A) \ V_{j_{n+1}}[A]) =
((\tilde{J}_n \ \nu \ A) \ V'_{j_1}[A] \ S'_{j_n}...S'_k ...
S'_{j_{n+1}-1}) \tr^* ((\gamma \ \nu \ a'_{j_n}) \ B'_{j_n})\
T_n.\]
Let $h_{n+1} = l_{j_n} + lg(B'_{j_n}:: T_n) + 1$ and
$\tilde{J}_{n+1} = \tilde{J}_n [\gamma := \l p \l x \l x_1...\l
x_{h_{n+1}} (x  \ (\gamma \ p \ x_1) \ ...$ $ \ (\gamma \ p \
x_{h_{n+1}}))]$. Then $((\tilde{J}_{n+1} \ \nu \ A) \
V_{j_{n+1}}[A]) \tr^* \l X_{j_{n+1}} ((\gamma \ \nu \ a_{j_{n+1}})
\ B_{j_{n+1}})$ for some term $a_{j_{n+1}}$ and  some sequence
$B_{j_{n+1}}$ of terms.

 \end{itemize}

\begin{com}\label{CO:rec}
 Note that, since the branch in ${\mathcal T}$ is recursive, it
 follows, by standard arguments, that the function $h$ defined by $h(i) =h_i$ is
 computable.
\end{com}

 \begin{defi}\label{J2}\hfill
\begin{itemize}
  \item  Let $H$ be  a $\l$-term that
 represents the function $h$.
  \item Let $T = \l a \l b \l c. (a   \    (b   \   (z  \   (suc \ k)  \  n  \
c)))$, $D = \l k \l n \l x.  (H   \   k   \    T    \    I \ x)$
and $\widehat{J} = (Y \  \l z. D)$ where $Y$ is the Turing fixed
point operator.
  \item For each $n \in \N$, we denote $(\widehat{J} \ c_n)$ by $J_n$.
\end{itemize}

\end{defi}

\begin{lem}\label{nu2}
For each $n \in \N$, $(J_n \ \n \ u) \tr^* \l y_1 ... \l y_{h_n}.
(u \ (J_{n+1} \ \n \ y_1) ... (J_{n+1} \ \n \ y_{h_n}))$.
\end{lem}

\begin{proof}
Easy.
\end{proof}

As in the previous section, we consider $J_n$ as constants with the reduction rules of the previous Lemma.
Let $A'=(J_0 \ \n \ A)$. We prove that $\n$ is persisting in $F[A']$.

\begin{defi}\hfill
\begin{enumerate}
  \item We define, for terms $u$, the sets $E_u$ of terms by the following
grammar:\\ $E_u = u \ | (J_n \ \nu \ e_u) \ | \l \overrightarrow{y}.
(e_u \ \overrightarrow{e_y})$
  \item Let  $t,t'$ be some terms. We denote by $t \leadsto t'$
  if there is a context $C$ with one hole such that $t=C[u]$ and
  $t'=C[u']$ where $u' \in E_u$.
\end{enumerate}
\end{defi}

\begin{lem}\label{commut2}
\begin{enumerate}
\item Assume $u \leadsto^* u'$ and $u \tr^* v$. Then,  $v
\leadsto^* v'$ for some $v'$ such that $u' \tr^* v'$.
\item Assume $u \leadsto^* u'$ and $u' \tr^* v'$. Then,  $v
\leadsto^* v'$ for some $v$ such that $u \tr^* v$.
\end{enumerate}
\end{lem}

\begin{proof}
Same proof as Lemmas \ref{commut-1} and \ref{commut-2}.
\end{proof}

\begin{lem}\label{hnfJ22}
If $u \tr^* \overrightarrow{\l x}. (x \ \overrightarrow{c})$ and
$u \leadsto^* u'$, then $u' \tr^* \overrightarrow{\l x}
\overrightarrow{\l y}.(x \ \overrightarrow{c'} \
\overrightarrow{e_y})$ for some $\overrightarrow{c}\leadsto^*
\overrightarrow{c'}$ and $\overrightarrow{e_y}$.
\end{lem}

\begin{proof}
This follows from Lemma \ref{commut2}
\end{proof}

\begin{defi}
Let $t$ be a solvable term. We say that:
\begin{enumerate}
  \item  $\n$ occurs nicely in
$t$ if the only occurrences of $\n$ are in a sub-term of the form
  $(J_n \ \n)$.
  \item $\n$ occurs correctly in
$t$ if it occurs nicely in $t$ and the head normal form of $t$
looks like $\overrightarrow{\l x}. (x \ \overrightarrow{c} \
\overrightarrow{e_y})$  for some final subsequence
$\overrightarrow{y}$ of  $\overrightarrow{x} $ of length at least
1 such that $\n$ does occur in $\overrightarrow{e_y}$.
\end{enumerate}

\end{defi}

\noindent The intuitive meaning of Lemma \ref{cles} below is ``for each $n
\in \N$, $t_{j_n} \leadsto^* (a_{j_n} \ B_{j_n})$ and $(A' \
V_{j_n}[A]) \tr^* \l X_{j_n}. (J_{n+1} \ \n \ a_{j_n} \
B_{j_n})$".

Strictly speaking, this is not true, because  the $a_{j_n}$ and
$B_{j_n}$ are not the ``real'' ones i.e. the  ones that occur in
the reduction with the real $\widehat{J}$.  For two reasons:

\begin{itemize}
  \item The first one is easily corrected:  in $a_{j_n}$ and  $B_{j_n}$ the
  constant $\g$ must be replaced by $J_{n+1}$
  \item The second one is more subtle. In the correct lemma, $\leadsto $ and the $J_n$  should be
  the ``real'' ones. But the
$a_{j_n}$ are defined using $\tilde{J}_p$ which are only fake
$J_p$. Stating the correct lemma  will need complicated, and
useless, definitions. We will not do it and thus   we state the
lemma  in the way it should be, intuitively, understood.

\end{itemize}

\begin{lem}\label{cles}
For each $n \in \N$, $t_{j_n} \leadsto^* (a_{j_n}[\gamma :=
J_{n+1}] \ B_{j_n}[\gamma := J_{n+1}])$ and $(A' \ V_{j_n}[A])
\tr^* \l X_{j_n}. (J_{n+1} \ \n \ a_{j_n}[\gamma := J_{n+1}] \
B_{j_n}[\gamma := J_{n+1}])$.
\end{lem}

\begin{proof}
By induction on $n$.
\end{proof}

\begin{prop}
For each $n \in \N$, $\n$ occurs correctly in $(A' \ V_{j_n}[A])$.
\end{prop}

\begin{proof}
We have $(A' \ V_{j_n}[A]) \tr^* \l X_{j_n}. (J_{n+1} \ \n \
a_{j_n}[\gamma := J_{n+1}] \ B_{j_n}[\gamma := J_{n+1}]) =$ \\$\l
X_{j_n}. ((\l z_1 ... \l z_{h_{n+1}}. (\ a_{j_n}[\gamma :=
J_{n+1}] \ (J_{n+2} \ \nu \ z_1) \ ... \  (J_{n+2} \ \nu \
z_{h_{n+1}})) \ B_{j_n}[\gamma := J_{n+1}])$. Since $h_{n+1} =
l_{j_n} + lg(B'_{j_n}T_n) + 1$, then, by Lemma \ref{cles}, $(A' \
V_{j_n}[A]) \tr^* \l X_{j_n} \l Z. (t'_{j_n} \ e_Z)$ where
$t'_{j_n} \leadsto^* t_{j_n}$, $e_Z \leadsto^* Z$ and $lg(Z) >
l_{j_n}$. Therefore, by Lemma \ref{hnfJ22}, $\n$ occurs correctly
in $(A' \ V_{j_n}[A])$.
\end{proof}

\begin{lem}\label{reste2}
Let $t$ be a solvable term. Assume that $\n$ occurs correctly in
$t$. Then $\n$ occurs (and it occurs correctly) in every reduct of
$t$.
\end{lem}

\begin{proof}
This follows immediately from the fact that if $(J_n \ \n \ y)
\leadsto^* u$, then $\n \in \beta\Omega(u)$.
\end{proof}

\begin{lem}
Let $k$ be an integer such that $\n$ occurs correctly in $(A' \
V_k[A])$. Then $\n$ occurs correctly in $(A' \ U_k[A'])$.
\end{lem}

\begin{proof}
This follows immediately from Lemmas \ref{hnfJ22} and
\ref{reste2}.
\end{proof}

\begin{prop}
$\n$ is persisting in $F[A']$.
\end{prop}

\begin{proof}
Same proof as Proposition \ref{p2}.
\end{proof}

\section{The other assumption}\label{5}

As we already said the fact that $\n$ is persisting in $F[A']$
does not imply that, letting  $A_n=A'[\n=c_n]$, $F[A_n] \not
\simeq
 F[A_m]$ for $n \neq m$. To ensure that the range
of $\l x.F$ is infinite,  we need another assumption on $F$.

Let $\l x.F$ be a closed term and  $A$ be such that $F[A] \not
\simeq F[\O]$. {\em In propositions \ref{6.1} and \ref{6.2}
below we assume that $F$ has the Barendregt's persistence property. Let $A'$ be
the corresponding term. We also assume that $A'$ has been obtained
by the way developed in section \ref{4.2} or in section
\ref{cas_2}}. Note that, in these cases, $\n$ is never applied in
a reduct of $F[A']$.

\begin{prop}\label{6.1}
Assume there is a sequence $(t_n)_{n \in \N}$ of distinct closed and
normal terms such that, for every $n$, $t_n$ never occurs as a
sub-term of some $t'$ such that $F[A'] \trbo^* t'$. Then, the
range of $\l x.F$ is infinite.
\end{prop}

\begin{proof}
Let $A_n=A'[\n:=t_n]$.  It is enough to show that $F[A_n] \not
\simeq F[A_m]$ for $n \neq m$. Assume $F[A_n] \simeq F[A_m]$ for
some $n\neq m$ and let $u$ be a common reduct. Since $\n$ is never
applied in a reduct of $F[A']$, there are reducts $a_n$ and $a_m$
of $F[A']$ such that $u=a_m[\n:=t_m]=a_n[\n:=t_n]$. This implies
that $t_n$ occurs in $a_m$. Contradiction.
\end{proof}

\noindent{\bf Remark}

Say that $u$ is a universal generator if, for every closed
$\l$-term $t$, there is a reduct of $u$ where $t$ occurs as a
sub-term. Before Plotkin gave his counterexample, Barendregt had
proved that the omega-rule is valid when $t,t'$ are not universal generators.
Our hypothesis on the existence of the sequence $(t_n)_{n \in \N}$
may look similar.  Assuming that $F[A']$ is not a universal
generator, there is a term $t$ that never occurs as a sub-term of
a reduct of $F[A']$. Letting   $t_0 = t$ and $t_{n+1} = \l x.
t_n$, the reducts of $F[A']$ never contain one of these terms.
Also, they are not equal, because otherwise $t$ is $\O$. This is
however not enough to show that the $F[\n:=t_n]$ are distinct
because we need that no reduct of $F[A']$ contains {\em a reduct}
of one of the $t_n$ (this is why, in our hypothesis,  we have
assumed that the $t_n$ are normal). This raises two questions :

\begin{enumerate}
  \item Say that a term $u$ is a weak generator if for every closed
$\l$-term $t$ one of its reducts occurs as a sub-term of a reduct
of $u$. A universal generator is, trivially, a weak generator. Is
the converse true ?
  \item Is it true that, if $F[A']$ is
a universal generator then so  is $F[A]$. Note that, somehow, the
$A'$ we have constructed is a kind of $\eta$-infinite expansion of
$A$.
\end{enumerate}

\noindent If these two propositions were true, we could replace the
assumption of Proposition \ref{6.1} by the (more elegant) fact
that $F[A]$ is not a universal generator.

\begin{defi}
Let $A_n=A'[\n=c_n]$.  Say that $F,A$ satisfy the Scope lemma if,
for every $n,m$, the fact that $F[A_n]\simeq F[A_m]$ implies that,
for some $k$, $(x \ U_k)[x:=A_n]\simeq (x \ U_k)[x:=A_m]$.
\end{defi}

The terminology ``Scope lemma'' is borrowed to A. Polonsky. In his
paper he stated an hypothesis (denoted as the scope Lemma) which
corresponds to the previous property.

\begin{prop}\label{6.2}
Assume $F,A$  satisfy the scope lemma. Then the range of $\l x.F$ is
infinite.
\end{prop}
\begin{proof}
Immediate.
\end{proof}

\section*{Acknowledgment}
We wish to thank Andrew
Polonsky for helpful discussions and also the anonymous referees
for their remarks and suggestions.

%% in general the use of bibtex is encouraged

\end{document}